\documentclass[11pt]{amsart}

\usepackage{amssymb}
\usepackage{amsmath}
\usepackage{amsfonts}
\usepackage{amsthm}
\usepackage{easy}
\usepackage{graphicx}
\usepackage[thinlines]{easymat}
\usepackage{epic,eepic}
\usepackage{bm}
\usepackage{yfonts}  
\usepackage{epsfig}
\usepackage{rotating}
\usepackage{blindtext}
\usepackage{enumerate}

\newcommand{\al}{\alpha}

\newcommand{\e}{\varepsilon}
\newcommand{\Or}{\mathcal{O}}

\newcommand{\ov}{\overline}

\newcommand{\gl}{\mathfrak{g}}
\newcommand{\ol}{\mathfrak{o}}
\newcommand{\pl}{\mathfrak{p}}
\newcommand{\s}{\mathfrak{s}}
\newcommand{\Ll}{\mathfrak{l}}
\renewcommand{\l}{\mathfrak{l}}
\renewcommand{\t}{\mathfrak{t}}

\newcommand{\ug}{\mathfrak{u}}
\newcommand{\gM}{\mathfrak{M}}

\newcommand{\F}{\mathbb F}
\newcommand{\Z}{\mathbb Z}

\DeclareMathOperator{\uni}{uni}
\DeclareMathOperator{\nil}{nil}

\DeclareMathOperator{\Ad}{Ad}
\DeclareMathOperator{\Lie}{Lie}

\DeclareMathOperator{\GL}{GL}

\DeclareMathOperator{\GU}{GU}

\DeclareMathOperator{\Atype}{A}
\DeclareMathOperator{\Btype}{B}
\DeclareMathOperator{\Ctype}{C}
\DeclareMathOperator{\Dtype}{D}
\DeclareMathOperator{\chara}{char}

\newtheorem*{theorem}{Theorem}
\newtheorem*{corollary}{Corollary}
\newtheorem*{proposition}{Proposition}
\newtheorem*{lemma}{Lemma}
\newtheorem*{example}{Example}
\newtheorem*{rmk}{Remark}
\newtheorem*{acknowledgements}{Acknowledgements}

\begin{document}

\title[Computing nilpotent and unipotent canonical forms]{Computing nilpotent and unipotent canonical forms: a symmetric approach}

\author[Matthew C. Clarke]{MATTHEW C. CLARKE\\
Trinity College,\\
Cambridge, CB\textup{2 1}TQ\\
email: \texttt{m.clarke@dpmms.cam.ac.uk}}

\maketitle

\begin{abstract}
Let $k$ be an algebraically closed field of any characteristic except $2$, and let $G = \GL_n(k)$ be the general linear group, regarded as an algebraic group over $k$. Using an algebro-geometric argument and Dynkin-Kostant theory for $G$ we begin by obtaining a canonical form for nilpotent $\Ad(G)$-orbits in $\gl\l_n(k)$ which is symmetric with respect to the non-main diagonal (i.e. it is fixed by the map $f : (x_{i,j})\mapsto (x_{n+1-j,n+1-i})$), with entries in $\{0,1\}$. We then show how to modify this form slightly in order to satisfy a non-degenerate symmetric or skew-symmetric bilinear form, assuming that the orbit does not vanish in the presence of such a form. Replacing $G$ by any simple classical algebraic group we thus obtain a unified approach to computing representatives for nilpotent orbits of all classical Lie algebras. By applying Springer morphisms, this also yields representatives for the corresponding unipotent classes in $G$. As a corollary we obtain a complete set of generic canonical representatives for the unipotent classes in finite general unitary groups $\GU_n(\F_q)$ for all prime powers $q$.\end{abstract}

\section{Introduction}

The Jordan canonical form for square matrices over an algebraically closed field $k$ can be thought of as a canonical form for conjugacy classes of the general linear group $G=\GL_n(k)$, or $\Ad(G)$-orbits of the general linear Lie algebra $\gl\l_n(k)$. More generally, for an element $g$ in an algebraic group $G$, we have the Jordan-Chevalley decomposition $g=g_sg_u=g_ug_s$ for a unique semisimple element $g_s$ and unipotent element $g_u$ in $G$, and for $x \in \gl = \Lie(G)$, $x= x_s + x_n$ for a unique semisimple element $x_s$ and nilpotent element $x_n$ in $\gl$ such that $[x_s, x_n] = 0$. This existence result does not, however, yield a method for finding a representative for these unique elements, up to the $G$-action, like the Jordan canonical form. A number of algorithms for obtaining representatives have been obtained, though, for unipotent and nilpotent elements. Since, in good characteristic (e.g. for root systems of Type $\Atype$ all characteristics are good, while for Types $\Btype$, $\Ctype$ and $\Dtype$, only 2 is bad, i.e. not good), there is a $G$-equivariant bijective morphism of varieties, a {\em Springer morphism}, from the nilpotent variety $\gl_{\nil}$ of $\gl$ to the unipotent variety $G_{\uni}$ of $G$ it follows that studying unipotent conjugacy classes is equivalent to studying nilpotent orbits. (In general, one needs to assume $G$ is simply connected for a Springer morphism to exist, which will always be the case for the groups in this paper.)

Gerstenhaber (\cite{Ger}) has given a method for computing representatives of nilpotent orbits in classical Lie algebras when the characteristic of the field is not $2$. In characteristic zero, Popov (\cite{Pop}) has described another one, which is a by-product of his determination of the strata of the nullcone of a linear representation of an algebraic group. The flavour of these approaches is quite different from the present one though and each relies on an analysis of the intrinsic classical root system, whereas in ours one only ever needs to consider the root system of Type $\Atype$. In this sense our approach stresses the relationship between nilpotent orbits in $\gl$ and those in the ambient $\gl\l_n(k)$. We also note that in \cite{DGra} De Graaf presents a probabilistic \textquoteleft trial and error\textquoteright\  algorithm for computing representatives of nilpotent orbits over any algebraically closed field. Whilst the parameters may be set so as to deliver arbitrarily large probability of success, the representative obtained is not canonical. Using a computer implementation De Graaf has computed representatives for all nilpotent orbits in exceptional Lie algebras.

Our main result (from which the other results are derived using comparatively less effort) is Theorem \ref{x_J}, which describes a canonical matrix form for nilpotent orbits in $\gl\l_n(k)$. The key features of this form are that it is upper triangular, symmetric with respect to the non-main diagonal (i.e. it is fixed by the map $f : (x_{i,j})\mapsto (x_{n+1-j,n+1-i})$), and its entries lie in $\{0,1\}$. The proof of this can be divided into roughly two phases. First we prove the existence of a non-canonical representative which enjoys certain nice properties. For this we use the Dynkin-Kostant classification of nilpotent orbits combined with an algebro-geometric argument. This will allow us to describe a calculus of \textquoteleft elementary operations\textquoteright\ that may be performed on the entries of this representative whilst remaining in the original orbit. These are then applied in the second phase of the proof to give an algorithm with the flavour of Gaussian elimination which puts the representative into the prescribed canonical form. The corresponding canonical forms for the other classical Lie algebras are then derived using short additional phases to this algorithm. We remark that the algorithms used to prove that these canonical forms are indeed representatives of the prescribed orbits are not needed to actually compute them, as is true of the Jordan canonical form. 

As we will see, this canonical form is sometimes not possible in characteristic 2. However, when this is the case we are still able to find a canonical form fixed by the composite of $f$ and the standard $q^{th}$-power Frobenius endomorphism, $F_q$. It turns out that this is sufficient to allow us to obtain canonical forms for the unipotent classes in $\GU_n(\F_q)$.

For a Frobenius endomorphism $F : G \rightarrow G$ there exists a corresponding Frobenius endomorphism on $\gl$, also denoted by $F$, which is compatible with a given Springer morphism [\cite{SpSt}, Theorem III.3.12]. In [\cite{Kawa}, \S 1.2] Kawanaka has given explicit formulas for $F$-stable Springer morphisms in the case of classical groups. Hence, when the characteristic of $k$ is good, the $G^F$-action on unipotent elements of $G^F$ agrees with the action on nilpotent elements of $\gl^F$. Therefore, if one is able to obtain a nilpotent element of $\gl$ which is fixed by $F$, then by applying such a map one may obtain an $F$-fixed unipotent element in $G$. In general, though, the $F$-fixed points of $F$-stable unipotent conjugacy classes in $G$ (resp. nilpotent orbits in $\gl$) split into several $G^F$-classes (resp. orbits). For $\GL_n(\ov{\F_p})$, however, this splitting does not occur for any Frobenius endomorphism (split or non-split). In Section \ref{springersection} we show how to obtain a canonical form for nilpotent elements in $\gl\l_n(k)$ which is fixed by $F$, thus yielding a canonical form for unipotent classes in the finite unitary groups. (For $F_q$ this is easy to see. Indeed, since $x \mapsto x + 1$ is a Springer morphism in this case, all entries in the image of one of our canonical elements lie in $\{0,1\}$ and are thus fixed by taking $q^{th}$-powers.)

\section{The general linear Lie algebra}

\subsection{A non-canonical representative} \label{SpSt}  In this section we set $G=\GL_n(k)$ and $\gl = \gl\l_n(k)$, but note that there is essentially no difference between this and the special linear case when one is concerned with nilpotent orbits or unipotent classes. (The situation is less straightforward in the finite setting.) We fix once and for all a nilpotent element $N \in \gl_{\nil}$ corresponding, by the Jordan canonical form, to some partition $\mu \vdash n$. We denote its $\Ad(G)$-orbit by $\Or_N$. Let $T$ and $B$ be the diagonal maximal torus and upper-triangular Borel subgroup of $G$ respectively, and denote the corresponding root system by $\Sigma$ and root spaces by $X_{\al}$ for $\al \in \Sigma$. Let $\t$ be the standard diagonal maximal torus of $\gl$, and, for $1 \le i \le n$, let $\e_i$ be the linear map $\t \rightarrow k$ which picks out the $i^{th}$ diagonal entry. We may then denote a set of positive roots by $\Sigma^+= \{ \e_i - \e_j \ |  1 \leqslant i < j \leq n \}$, such that $X_{\e_i-\e_j}$ is the root space consisting of all matrices in $\gl$ which are zero except for their $(i,j)^{th}$ entry. The corresponding simple roots may then be written as $\Pi = \{\al_1, \dots, \al_{n-1}\}$, where $\al_i = \e_i-\e_{i+1}$.

We present the main results from the Dynkin-Kostant-Springer-Steinberg theory as follows (see, e.g., [\cite{Kawa}, pp. 177--178] and the references there). (Note that, since we will only apply this in the case where $\Sigma$ is of Type $A$, an elementary proof is possible, by following, e.g., [\cite{Lus5}, \S 2].)

\begin{theorem} With the above set-up, there exists a $\Z$-grading \[\gl = \bigoplus_{i \in \Z}\gl_i\]

\noindent of $\gl$, depending only on $\Or_N$, with the following properties.

\begin{enumerate}[{\em (i)\ }]
	\item Each $\gl_i$ is a sum of root spaces.
	
	\item We may assume {\em (}by replacing by a conjugate if necessary{\em )} that $N \in \gl_2$.  
	
	\item $\pl_N = \bigoplus_{i \geq  0}\gl_i$ is the Lie algebra of a standard {\em (}block upper-triangular{\em )} parabolic subgroup $P_{N}$ of $G$.

	\item $\Ll_N = \gl_0$ is the Lie algebra of the block-diagonal Levi subgroup $L_N$ of $P_N$.
	
	\item For $i \geq  1, \ug_{N,i} = \bigoplus_{j \geq  i}\gl_i$ is the Lie algebra of a connected normal unipotent subgroup $U_{N,i}$ of $P_N$. In particular, $U_{N,1}$ is the unipotent radical of $P_N$.

	\item Each $\gl_i$ is $\Ad(L_N)$-stable.
	
	\item $\Or_N \cap \gl_2$ is dense in $\gl_2$.
	
\item There exists a unique additive function $h_N : \Sigma \rightarrow \Z$, fixed by the non-trivial graph automorphism of the Dynkin diagram of $G$, such that \begin{enumerate}[{\em (a)\ }]
	
	\item $h_N(\al) \in \{0,1,2\}$ for each $\al \in \Pi$;

	\item $\gl_{i} = \bigoplus_{h_N(\al)=i}X_{\al}.$
\end{enumerate}
\end{enumerate}

\end{theorem}

The Dynkin diagram with nodes labelled by the numbers $h(\al_i)$, corresponding to the simple roots, is called the {\em weighted Dynkin diagram associated to} $\Or_N$. We may partition $\Sigma$ into the following subsets. For $i\in \Z$ set $\Sigma_i = \{ \al \in \Sigma \ | \ h(\al)=i \} = \{ \al \in \Sigma \ | \ X_{\al} \subseteq \gl_i \}$.

We explicitly construct the function $h$ as follows. Let $\mu = (\mu_1 \geq  \mu_2 \geq  \cdots \geq  \mu_r)$. Then for each $\mu_i$, consider \[Y_i = \{\mu_i - 1, \mu_i - 3, \dots, 3 - \mu_i, 1 - \mu_i\}. \] \noindent Viewing $Y=\coprod_iY_i$ as a multiset of $n$ integers, arrange in decreasing order: \[Y = \{\nu_1 \geq  \nu_2 \geq  \cdots \geq  \nu_n \}.\] \noindent Then we define $h$ on $\Pi$ by putting $h(\al_i) = \nu_i - \nu_{i+1}$. This uniquely determines $h$.

In [\cite{Sho}, \S 2], it is shown how to construct $\Sigma_1$. Generalising this we construct $\Sigma_2$. Let $\Pi_{\e}$ be the set of simple roots with $h$-weight $\e$ for ${\e} = 1, 2$. For $\al_i \in \Pi_1$ (and, later, $\Pi_2$), let $a_i$ be the smallest integer such that $a_i > i$ and $h(\al_{a_i}) \not= 0$, and let $b_i$ be the largest integer such that $b_i < i$ and $h(\al_{b_i})\not= 0$. Then we obtain rectangular subsets \[\Psi_i = \{\e_s - \e_t \ | \  b_i + 1 \leq s \leq i, i + 1 \leq t \leq   a_i \}.\]

\noindent (If $a_i$ (resp. $b_i$) does not exist, then set $a_i=n$ (resp. $b_i=0$).) Then, as observed in [\cite{Sho}, \S 2], we have a disjoint union \[\Sigma_1 = \coprod_{\al_i \in \Pi_1}\Psi_i.\]

Now we construct $\Sigma_2$. A pair $\Psi_i, \Psi_j \subseteq \Sigma_1$ are said to be {\em adjacent} if $h(\al_k)=0$ whenever $i < k < j$. We will also say that $\al_i$ and $\al_j$ are adjacent when this is the case. For each adjacent pair $\Psi_i$, $\Psi_j$ we define another subset $\Psi_{i,j}$ of $\Sigma^+$ by  \[\Psi_{i,j} = \{ \e_s - \e_t  \ | \  b_i +1 \leq   s \leq   i, j + 1 \leq   t \leq     a_j \}.\]

\noindent Then we have another disjoint union \[\Sigma_2 = \left( \coprod_{\alpha_i, \alpha_j \in \Pi_1}\Psi_{i,j}\right) \coprod \left(\coprod_{\al_k \in \Pi_2}\Psi_k\right),\]

\noindent where the first union is taken over adjacent pairs. We shall call the $\Psi_{i,j}$ and $\Psi_k$ appearing in the above decomposition the {\em blocks} of $\gl_2$. 

We now describe $\gl_2$. Define sequences $l_1\geq  l_2 \geq  l_3 \geq  \dots$ and $k_1\geq  k_2 \geq  k_3 \geq  \dots$ (related to the dual partitions of the purely odd and purely even parts of $\mu$) as follows. For $i \geq  1$ set  \[l_i = \#\{  \mu_j \mbox{ odd } |\  \mu_j \geq  2i - 1 \},\]

\noindent and \[k_i = \#\{ \mu_j \mbox{ even } |\  \mu_j \geq  2i \}.\]

\begin{lemma}   $\gl_2$ consists of matrices of the form $x$, satisfying the following properties.

\begin{figure}[ht]		\centering

\includegraphics{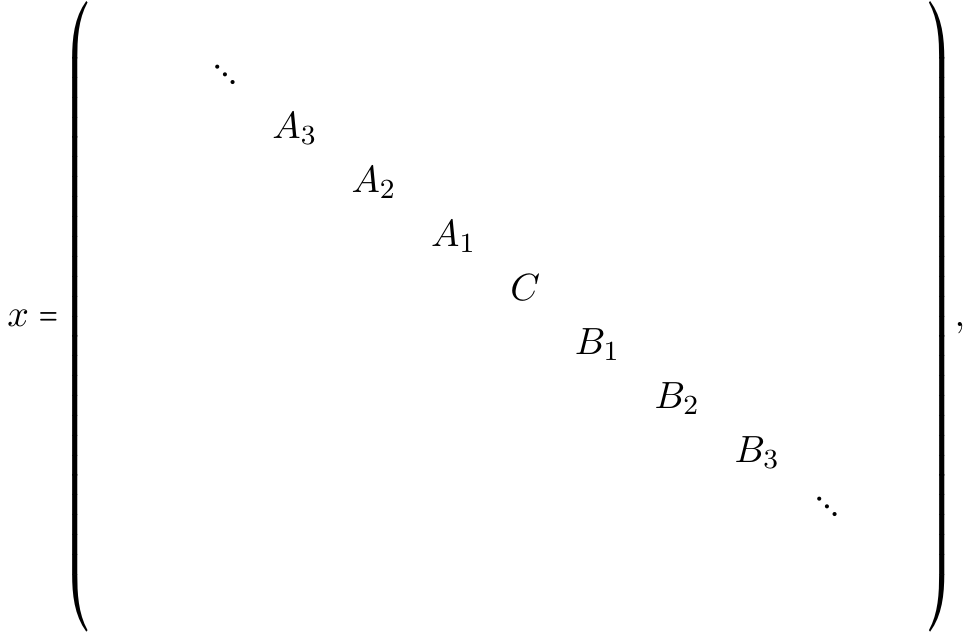}

\end{figure}

\begin{enumerate}[{\em (i)\ }]
	\item  \label{l1} All entries are zero outside the rectangular blocks $A_i, B_i, C$, and these blocks correspond to the blocks of $\gl_2$. (I.e. they are spanned by the root vectors for the roots occurring in each of the blocks of $\gl_2$.)
	\item  \label{l2} The entries $x_{i,j}$ inside the blocks are arbitrary, and for all such entries $i \leq j$.
	\item  \label{l3} The block structure is symmetric with respect to the non-main diagonal, i.e. it is fixed by the map $f : (x_{i,j})\mapsto (x_{n+1-j,n+1-i})$.
	\item  \label{l4} Each row (resp. column) intersects at most one block.
	\item  \label{l5} The middle block $C$ {\em (}which is necessarily square by the above{\em )} exists if, and only if, some $\mu_i$ is even.
	\item  \label{l6} The set of blocks may be partitioned into subsets $I=\{\dots, A_{i_2}, A_{i_1}, B_{i_1}, B_{i_2}, \dots \}$ and $J=\{\dots, A_{j_2}, A_{j_1}, C, B_{j_1}, B_{j_2}, \dots \}$ with the following properties:
\begin{enumerate}[{\em (a)\ }]
	\item $J=\emptyset$ if $C$ does not exist.
	\item Each of $I$ and $J$ is a symmetric block structure with respect to the non-main diagonal.
	\item $A_{i_r}$ is an $l_{r+1} \times l_{r}$ matrix {\em (}and hence  $B_{i_r}$ is an $l_{r} \times l_{r+1}$ matrix{\em )}.
	\item $A_{j_r}$ is a $k_{r+1} \times k_{r}$ matrix {\em (}and hence  $B_{j_r}$ is a $k_{r} \times k_{r+1}$ matrix{\em )}.
	\item Set $C = A_{j_0} = B_{j_0}$, $A_{i_0} = B_{i_1}$ and $B_{i_0} = A_{i_1}$. Then row $k$ of $x$ intersects $A_{i_r}$ {\em (}resp. $A_{j_{r}}, B_{i_{r}}, B_{j_{r}}${\em )} if, and only if, column $k$ intersects $A_{i_{r+1}}$ {\em (}resp. $A_{j_{r+1}}, B_{i_{r-1}}, B_{j_{r-1}}${\em )}.
	
	\end{enumerate}
\end{enumerate}

\end{lemma}

\begin{proof} The first four parts are clear from the construction. Observe that (\ref{l5}) is equivalent, by symmetry, to there existing a root $\al$ of weight $2$ which is fixed by the non-trivial graph automorphism $\Pi \rightarrow \Pi$.  If all $\mu_i$ are even then it is easy to see that such a root exists in $\Pi$, i.e. the central node of the Dynkin diagram. If at least one, but not all, $\mu_i$ are even then the middle part of the sequence of weights of $\Pi$ is of the form $1, 0, \dots, 0, 1$, with the two $1$s equidistant from the centre. Then $\al$ may be taken to be the sum of the roots corresponding to these $1$s and the intervening $0$s. If all the $\mu_i$ are odd then the middle part of the sequence of weights of $\Pi$ is of the form $2, 0, \dots, 0, 2, \dots$, with the two $2$s equidistant from the centre. Thus, a root fixed by the graph automorphism cannot have weight $2$.

We will now construct the sets $I$ and $J$. If $\mu$ has both odd and even parts choose $i$ to be minimal such that $\mu_i - \mu_{i+1}$ is odd. Then \[Y_i \cup Y_{i+1} = \{ \mu_i - 1, \mu_i - 3, \dots, \mu_i - (\mu_i - \mu_{i+1}), \mu_{i+1} - 1, \mu_{i+1} - 2, \mu_{i+1} - 3, \dots,\] \[3 - \mu_{i+1}, 2 - \mu_{i+1}, 1 - \mu_{i+1}, (\mu_i - \mu_{i+1}) - \mu_i, \dots, 3 - \mu_i, 1 - \mu_i \}. \]

\noindent It follows that each integer $k$, such that $\mu_{i+1} \geq  k \geq  -\mu_{i+1}$, appears with non-zero multiplicity in $Y$, and that an integer $k$ such that $\mu_{1} - 1 \geq  k \geq  \mu_{i+1}$, appears with non-zero multiplicity in $Y$ if, and only if, it has the same parity as $\mu_1 - 1$. We thus obtain a symmetric partition of $\Pi$ into three parts as follows. \[\Pi = \{\al_1, \dots, \al_{\hat{\iota}} \} \cup \{ \al_{\hat{\iota} + 1}, \dots, \al_{n - \hat{\iota}} \} \cup  \{ \al_{n + 1 - \hat{\iota}}, \dots, \al_{n - 1} \}, \]

\noindent where $\hat{\iota}$ is the largest number such that $\hat{\iota} < (n-1)/2$ and $h(\al_{\hat{\iota}}) = 2$. Assume for now that $\mu$ does not consist only of even parts. Then the blocks of $\gl_2$ of the form $\Psi_i$, which are in bijection with the simple roots of weight 2, may be split into two sets of adjacent blocks corresponding to the end parts of the above partition. The blocks of the form $\Psi_{i,j}$ are in bijection with sets of adjacent roots of weight 1 from the middle part. If $\mu$ does consist only of even parts then only blocks of the form $\Psi_i$ occur in $\gl_2$.

     Let $\al_{m_1}, \al_{m_2}, \al_{m_3}, \dots$ denote the elements of $\Pi_1$ with $m_1 \leq m_2 \leq m_3 \leq \cdots$, and set \[A = \{ \Psi_i \ | \ \al_k \in \Pi_2 \} \cup \Psi_{m_1,m_2} \cup \Psi_{m_3,m_4} \cup \Psi_{m_5,m_6} \cup \cdots\] 

\noindent and \[B = \Psi_{m_2,m_3} \cup \Psi_{m_4,m_5} \cup \Psi_{m_6,m_7} \cup \cdots.\]

\noindent It is clear that $A$ and $B$ partition the roots which determine $\gl_2$ and that each is a union of blocks. Set $\{A,B\} = \{I,J\}$ so that $J$ contains the central block. Then (\ref{l6}) follows from this construction. (See also the example below.) \end{proof}

\begin{example} If $\mu$ has only odd or only even parts then the elements of $Y$, disregarding multiplicities, are $\mu_1 - 1, \mu_1 - 3,  \dots , 3 - \mu_1, 1 - \mu_1$ and therefore the only weights that can occur are $\{0,2\}$. In this case all blocks are of the form $\Psi_i$ and we should set $I = \nolinebreak \{ \al \in \nolinebreak \Sigma\ |\ h(\al)=2\}$, $J=\emptyset$ if all parts are odd and vice versa if all parts are even.\end{example}

Given an $n \times n$ matrix $x$ and a union of blocks $X$, we will denote by $x_X$ the matrix obtained by replacing all entries of $x$ which do not correspond to $X$ by zero. If $X$ is a block, we will also refer to $x_X$ as a {\em block of } $x$. Our eventual canonical form will be an element $x \in \Or_N \cap \gl_2$, with entries in $\{0,1\}$, such that $f(x_{A_i}) = x_{B_i}$ for $i\geq  1$, and $f(x_C) = x_C$ if $C$ exists. By the density of $\Or_N \cap \gl_2$ in $\gl_2$, we may obtain a representative in any non-empty open set. In what follows, we view $\gl_2$ as an affine space in its own right, and therefore ignore coordinates of $\gl$ outside of $\gl_2$. We will use the following open set.

Let $\gM = 2^{\dim \gl_2 }{\dim \gl_2 }^{1/2}$ and then consider all polynomials $f_i$ ($i \in I$, some indexing set) in $\dim \gl_2$ indeterminates of degree at most $\gM$, with coefficients in $\{0,\pm 1\}$. Define $V = V(f_i)$ to be the zero locus of the $f_i$ and let $S= \gl_2 \setminus V$ be the corresponding open set in $\gl_2$. Hence, by density, we may choose a representative $x \in \Or_N \cap S$. (The number $\gM$ might seem rather arbitrary here but its nature will become clear later when we consider an algorithm for putting $x$ into symmetric form.)

\subsection{Making $x$ canonical} In what follows we assume that $x \in \Or_N \cap S$ is chosen in the manner explained above, and fixed. We now move on to the second phase of the proof of Theorem \ref{x_J}. This is based on the fact that the $L$-orbit of $x$ is contained in $\Or_N \cap \gl_2$. For this we will need some new terminology. We shall refer to the {\em rows} and {\em columns} of the blocks $A_i$ and $C$ as those inherited from the ambient matrix, but it will be convenient for us to invert this definition for the blocks $B_i$ (for $i \ge 1$).

\begin{lemma} For any block $X$ and any elementary row or column operation on $x_X$, there exists $l \in L$ such that conjugation by $l$ on $x$ agrees with this operation. If $X = A_{i_r}$ {\em (}resp. $A_{j_r}, B_{i_r}, B_{j_r}${\em )}, then, for row operations, $l$ can be chosen so that it acts trivially on all other blocks except $x_{A_{i_{r+1}}}$ {\em (}resp. $x_{A_{j_{r+1}}}, x_{B_{i_{r+1}}}, x_{B_{j_{r+1}}}${\em )}, and, for column operations, so that it acts trivially on all other blocks except $x_{A_{i_{r-1}}}$ {\em (}resp. $x_{A_{j_{r-1}}}, x_{B_{i_{r-1}}}, x_{B_{j_{r-1}}}${\em )}. Furthermore, these pairs of actions are described in Tables \ref{dualoperations1} and \ref{dualoperations2}. \end{lemma}

\begin{table}[ht]	   \centering   	

\caption{Duality of operations for $x_I$ (general case)} 

\begin{tabular}{cc}
\hline
{\bf row operation on $x_{A_{i_r}}$} & {\bf column operation on $x_{A_{i_{r+1}}}$}  \\
{\bf (resp. $x_{A_{j_{r}}}, x_{B_{i_{r}}}, x_{B_{j_{r}}}$)} & {\bf (resp. $x_{A_{j_{r+1}}}, x_{B_{i_{r+1}}}, x_{B_{j_{r+1}}}$)} \\ [-2.5ex] & \\  \hline 
swap rows $a$ and $b$ & swap columns $a$ and $b$ \\ \hline
multiply row $a$ by $\lambda$ & multiply column $a$ by $\lambda^{-1}$ \\ \hline
add $\lambda$ times row $a$ to row $b$ & add $-\lambda$ times column $b$ to column $a$ \\ \hline
\end{tabular}

\label{dualoperations1}

\end{table}

\begin{table}[ht]	   \centering   	

\caption{Duality of operations for $x_I$ (special case)}     

\begin{tabular}{cc}
\hline
{\bf column operation on $x_{A_{i_1}}$} & {\bf column operation on $x_{B_{i_1}}$}  \\ [-2.5ex] & \\  \hline
swap columns $a$ and $b$ & swap columns $a$ and $b$ \\ \hline
multiply column $a$ by $\lambda$ & multiply column $a$ by $\lambda^{-1}$ \\ \hline
add $\lambda$ times column $a$ to column $b$ & add $-\lambda$ times column $b$ to column $a$ \\ \hline 
\end{tabular}

\label{dualoperations2}

\end{table}

\begin{proof} It is clear from Lemma \ref{SpSt} that there exists such an elementary matrix $l$ in $G$. The proof follows from the observation that the non-diagonal, non-zero entries of $l$  correspond to roots of weight zero; thus $l \in L$ by Theorem \ref{SpSt}.\end{proof}

This allows us to consider $x_I$ and $x_J$ separately.

\subsection{Symmetrising $x_I$}  \label{rankremark}

First consider the central pair of blocks of $x_I$. Rather than using the column numbering from the ambient matrix, we shall translate this for ease of notation. Our set-up is as in Figure \ref{centralpair}.

\begin{figure}[ht]		\centering

\includegraphics{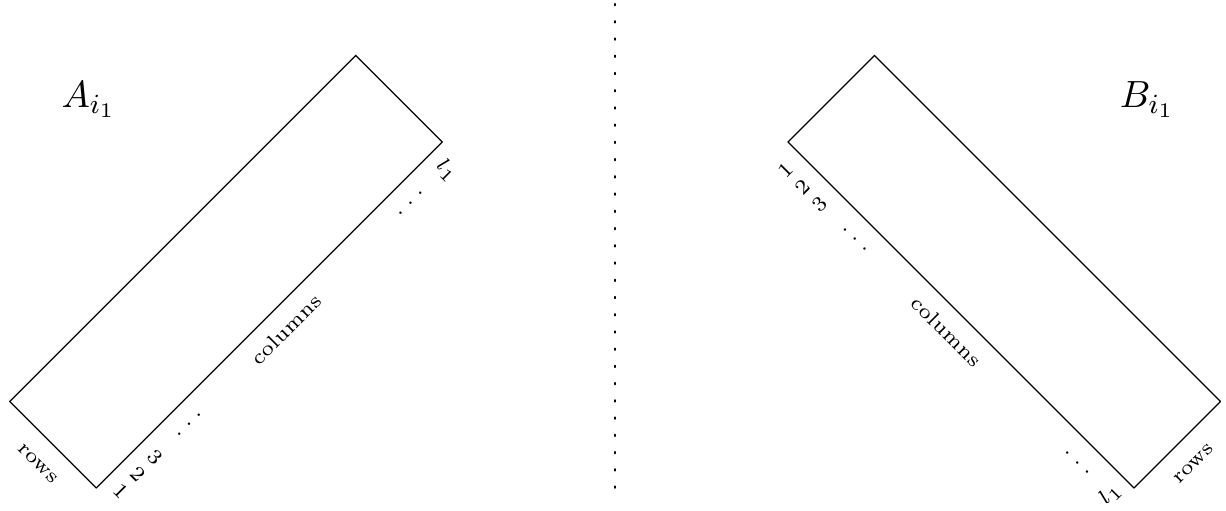}

\caption{Central pair of blocks in $x_I$}

\label{centralpair}

\end{figure}

For $m \ge 1$, we define an $m \times m$ matrix $J_m$ as follows: $(J_m)_{i,j} = 1$ if $i+j = m + 1$, and $0$ otherwise.

Now the column operations are in duality as in Table \ref{dualoperations2}. Using these dual operations, together with arbitrary row operations, we may obtain Figure \ref{centralpair2}, where the dotted lines in the blocks denote diagonal arrays of $l_2$ ones and the blank space zeros. This is achieved as follows.

\begin{enumerate}
	\item Perform Gauss-Jordan elimination to put $A_{i_1}$ in the desired form. \label{one}
	
	\item Perform row operations until the rightmost square of $B_{i_1}$ is $J_{l_2}$. \label{two}
	
	\item Using column operations in $B_{i_1}$, delete all entries not in this rightmost square.

\end{enumerate}

\begin{figure}[ht]		\centering

\includegraphics{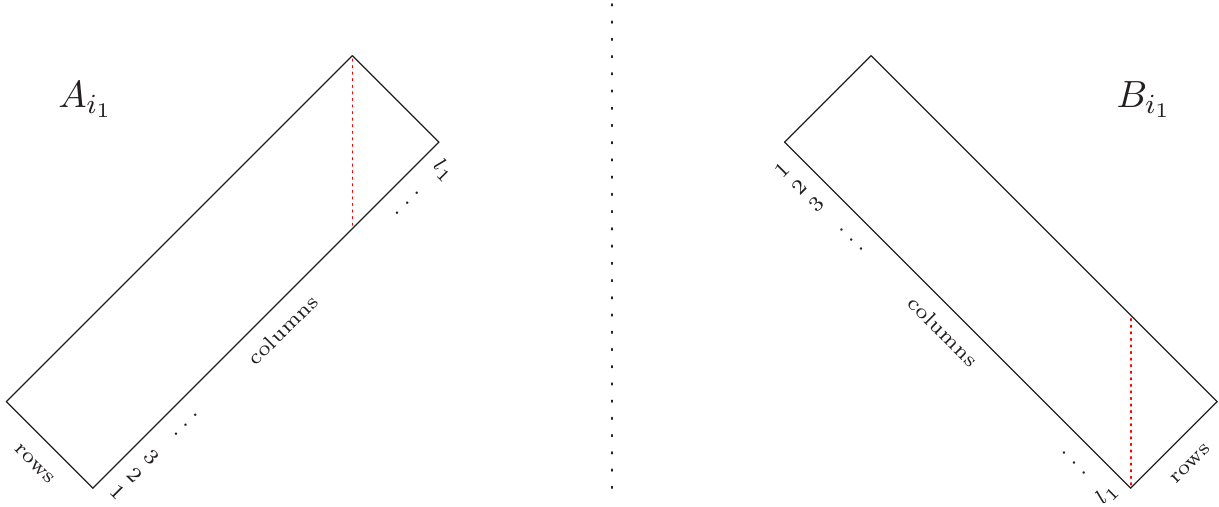}

\caption{Central pair of blocks in $x_I$} 

\label{centralpair2}

\end{figure}


\begin{rmk}  Implicit in \ref{two} is that the rank of the rightmost $l_2 \times l_2$ sub-matrix of $B_{i_1}$ has remained maximal (i.e. equal to $l_2$) throughout \ref{one}. This is valid because of the way $S$ was constructed. Indeed, first note that the initial representative $x$ has this property, or else a determinant polynomial will be satisfied. Proceeding by induction, assume that we have completed a certain number of steps of the Gauss-Jordan algorithm, and denote the resulting $B_{i_1}$-component by $x_{B_{i_1}}$ and the resulting rightmost $l_2 \times l_2$ sub-matrix of $B_{i_1}$ by $x_Z$. Let the inductive hypothesis be that the determinant of every square sub-matrix of $x_{B_{i_1}}$ may be written as a Laurent polynomial in the entries of the initial representative $x$, with coefficients in $\{ 0, \pm 1 \}$. Letting $x'_Z$ denote our sub-matrix after one more elementary column operation on $A_{i_1}$, we may write \begin{equation}  \det x'_Z =  \det x_Z + \lambda \det x_Y, \end{equation}

\noindent if the operation results in a column from outside $x_Z$ being added into $x_Z$, where $x_Y$ is another sub-matrix of $x_{B_{i_1}}$ and $\lambda$ is as in Table \ref{dualoperations2}. Or, \begin{equation}  \det x'_Z =  \lambda \det x_Z, \end{equation}

\noindent if we encounter an internal column operation on $x_Z$. Then one checks that $\lambda$ is either $\pm 1$ or a product of entries from $A_{i_1}$ and their inverses, up to sign. It follows that $\det x'_Z$ is a Laurent polynomial in the entries of the initial representative $x$, by the inductive hypothesis. If this vanishes then we may construct a polynomial in the entries of the initial representative $x$, with coefficients in $\{ 0, \pm 1 \}$ which also vanishes. This contradicts our choice of set $S$, since one may check that the degree of the polynomial will be lower than the bound $\gM$.  \end{rmk}

Next, if $l_2$ is even, then move the columns corresponding to the left half of the copy of $J_{l_2}$ in $A_{i_1}$ to the far left of $A_{i_1}$ using swapping operations. Together with the dual actions on $B_{i_1}$, we obtain the symmetric Figure \ref{centralpair3}, where the dots denote a diagonal array of $l_2/2$ $1$s.

\begin{figure}[ht]		\centering

\includegraphics{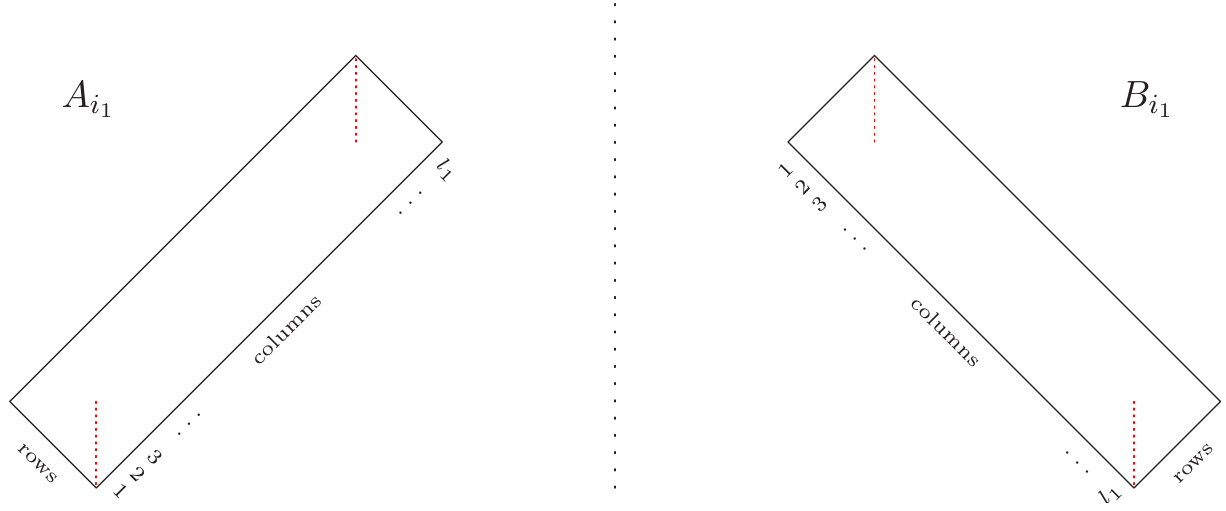}

\caption{Symmetrised central pair of blocks in $x_I$ ($l_2$ even)}

\label{centralpair3}

\end{figure}

If $l_2$ is odd, the above step will clearly not work. In this case, we do not seek symmetry immediately. Rather, perform the column swaps using the leftmost $(l_2-1)/2$ columns of $I$ in $A_{i_1}$. The result will be asymmetric, but the only asymmetries will be that column $l_1 + 1 - (l_2+1)/2$ of $A_{i_1}$ has a $1$ in the middle while column $(l_2+1)/2$ of $B_{i_1}$ consists of $0$s, and vice versa, as in Figure \ref{centralpair4}, where the filled circle denotes a $1$ and the blank circle a $0$.

\begin{figure}[ht]		\centering

\includegraphics{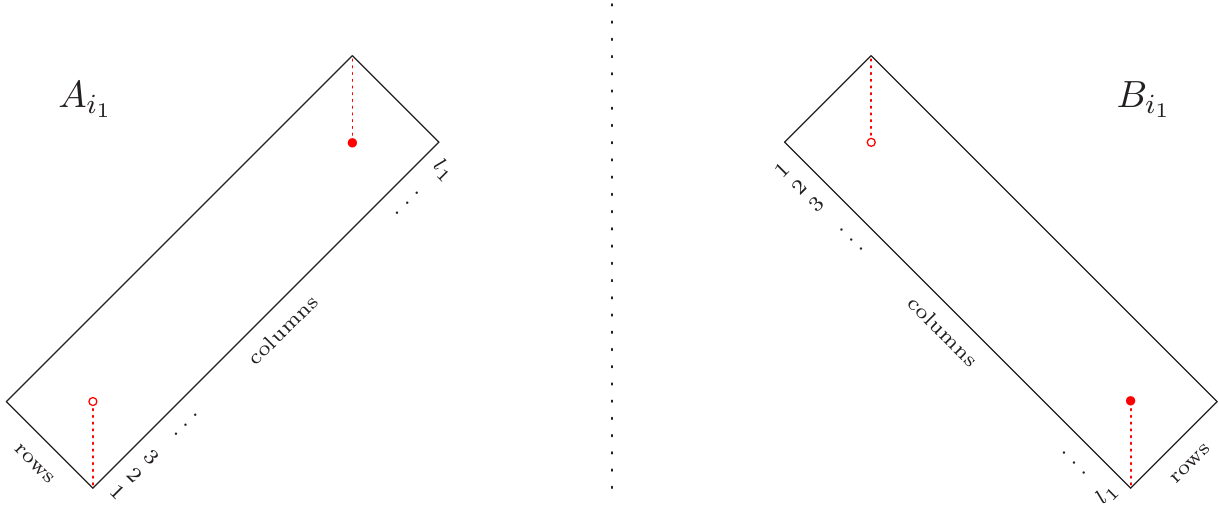}

\caption{Central pair of blocks in $x_I$ ($l_2$ odd)} 

\label{centralpair4}

\end{figure}

For $x_I$ in general we consider the sequence of pairs \begin{equation} \label{pairsequence}(A_{i_1},B_{i_1}), (A_{i_2},B_{i_2}), (A_{i_3},B_{i_3}), \dots\end{equation} in order. I.e. we iteratively move out from the centre. We now explain how each of these can be reduced to the case already dealt with. Any elementary column operation on a block in a new pair $(A_{i_k},B_{i_k})$ will inevitably induce a row operation on its neighbour (with subscript $i_{k-1}$, currently described by Figure \ref{centralpair3} or \ref{centralpair4}), thus knocking it out of canonical form. However, there exists a single elementary column operation on the latter which rectifies this. Then we apply the same process to its neighbour and so on until we reach the other member of $(A_{i_k},B_{i_k})$. This creates a duality of operations on $(A_{i_k},B_{i_k})$ which agrees with Table \ref{dualoperations2}, and so we reduce to the central pair case without damaging the pairs of blocks in between. It might be helpful to view the intervening blocks as a mirror along which one reflects. Eventually, all pairs of blocks will be as in Figure \ref{centralpair3} or \ref{centralpair4}. To achieve overall symmetry, we must now address those of the form Figure \ref{centralpair4}.

For simplicity of notation consider the central pair case first. We obtain the symmetrised form, Figure \ref{centralpair5}, by the following sequence of operations (together with their duals).

\begin{enumerate}
	\item Add column $l_1 + 1 - (l_2+1)/2$ to column $(l_2 + 1)/2$ in $A_{i_1}$.
	
	\item Add $1/2$ times column $l_1 + 1 - (l_2+1)/2$ to column $(l_2 + 1)/2$ in $B_{i_1}$.
	
	\item Multiply row $(l_2+1)/2$ of $B_{i_1}$ by $2$.
	
	\item Multiply column $l_1 + 1 - (l_2+1)/2$ of $A_{i_1}$ by $2$.
	
	\item One may also need to multiply the central rows of $A_{i_2}, A_{i_3}, A_{i_4},\dots$ by $2$ depending on the parity of the $l_1, l_2, l_3, \dots$.
\end{enumerate}

The same method works for pairs as in Figure \ref{centralpair4} in general, however, in order to exploit the mirror property at each stage we must take the sequence (\ref{pairsequence}) in the opposite order. This results in the desired canonical form for $x_I$.

\begin{figure}[ht]		\centering

\includegraphics{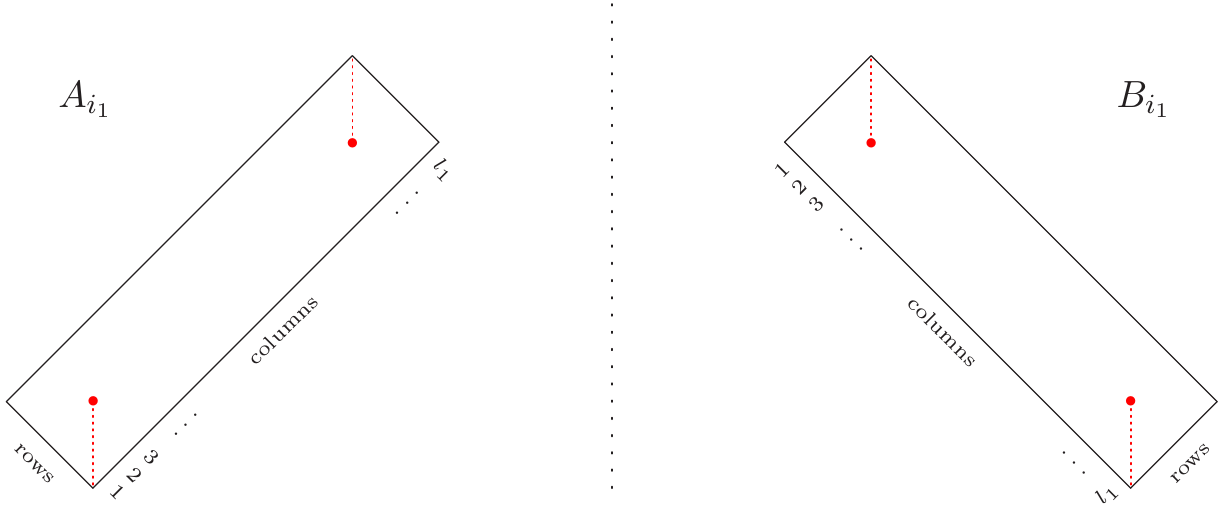}

\caption{Symmetrised central pair of blocks in $x_I$ ($l_2$ odd)} 

\label{centralpair5}

\end{figure}

\subsection{Symmetrising $x_J$}  \label{x_J}

A slightly ugly part of the symmetrising algorithm in the last section was the fact that asymmetric pairs, as in Figure \ref{centralpair4}, may be part of the central mirror arrangement that is built up, necessitating a second phase to the algorithm. The algorithm for symmetrising $x_J$ also uses a mirror arrangement of intervening blocks, but symmetrising $x_J$ is more straightforward as only one phase is needed. However, the presence of a central square block requires a slightly different calculus of dual operations.

First use arbitrary operations to transform $C$ into $J_{k_1}$, as in Figure \ref{centralblock1}. Because of this, for any elementary operation on the left of the central block, there exists an elementary operation on the right which cancels it out, and vice versa. Combining this duality with the usual duality on adjacent blocks described in Table \ref{dualoperations1}, Table \ref{dualoperations3} gives a duality of column operations on the pair $(x_{A_{j_1}}, x_{A_{j_1}})$ in $x_J$. Notice that, because of the way we have labelled the columns in Figure \ref{centralblock1}, Tables \ref{dualoperations2} and \ref{dualoperations3} are identical. Now we will show how to obtain Figure \ref{centralblock1}, which acts as a mirror for operations on $(x_{A_{j_2}}, x_{A_{j_2}})$.

\begin{table}[ht]	     \centering   	

\caption{Duality of operations for $x_J$}          

\begin{tabular}{cc}
\hline
{\bf column operation on $x_{A_{j_1}}$} & {\bf column operation on $x_{B_{j_1}}$}  \\ [-2.5ex] & \\  \hline  
swap columns $a$ and $b$ & swap columns $a$ and $b$ \\ \hline
multiply column $a$ by $\lambda$ & multiply column $a$ by $\lambda^{-1}$ \\ \hline
add $\lambda$ times column $a$ to column $b$ & add $-\lambda$ times column $b$ to column $a$ \\ \hline 
\end{tabular}

\label{dualoperations3}

\end{table}

Using Table \ref{dualoperations3}, we may symmetrise the pair $(x_{A_{j_1}}, x_{B_{j_1}})$ using the following operations (and their duals). One may check that the result is described by Figure \ref{centralblock1}.

\begin{enumerate}

	\item Put $x_{A_{j_1}}$ in the desired form using column operations.
	
	\item Obtain the identity matrix on the leftmost part of $x_{B_{j_1}}$ using row operations.
	
	\item Add suitable scalar multiples of the columns of this identity matrix to eliminate the rest of $x_{B_{j_1}}$.

\end{enumerate}

\begin{figure}[ht]		\centering

\includegraphics{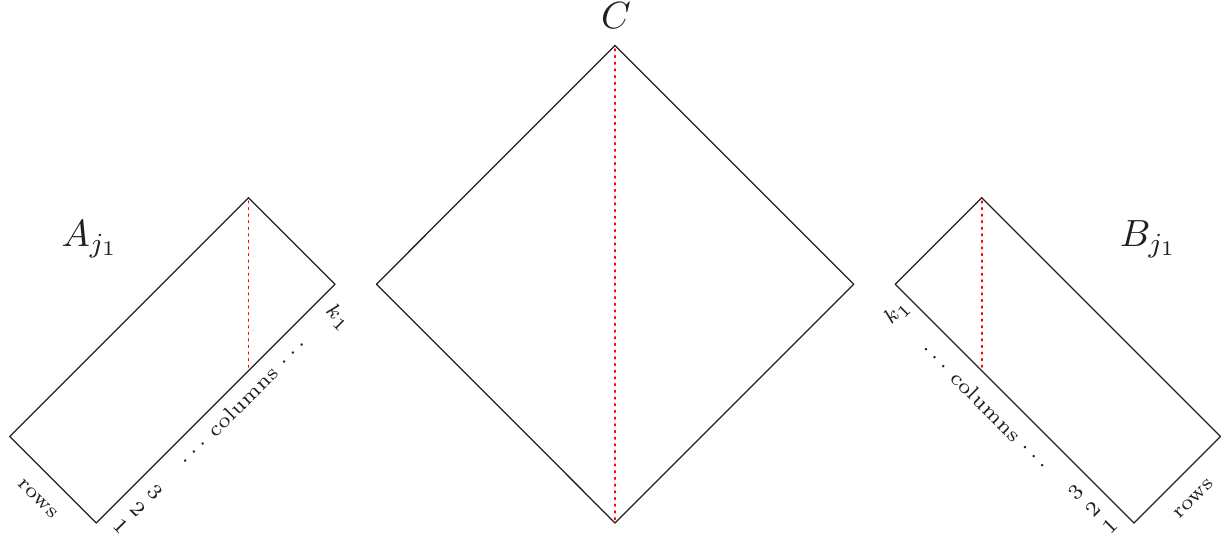}

\caption{Symmetrised central arrangement in $x_J$} 

\label{centralblock1}

\end{figure}

It is clear that this configuration allows Table \ref{dualoperations3} and the above algorithm to be extended to each of the pairs in the sequence \begin{equation} \label{pairsequence2}(A_{j_1},B_{j_1}), (A_{j_2},B_{j_2}), (A_{j_3},B_{j_3}), \dots\end{equation} in order. This completes the proof of our main result, which we now state precisely.

\begin{theorem} Let $\mu \vdash n$, and let $\Or_{\mu} \subset \gl\l_n(k)$ denote the nilpotent orbit corresponding, via the Jordan canonical form, to $\mu$. Considering the blocks described by Lemma \ref{SpSt} let $x$ denote the following matrix:

\begin{enumerate}[{\em (i)\ }]
	\item \label{mainone} The entries of $x$ agree with Figure \ref{centralpair3} if $l_2$ is even, and Figure \ref{centralpair5} if $l_2$ is odd, on blocks $A_{i_1}, B_{i_1}$ and Figure \ref{centralblock1} on blocks $A_{j_1}, B_{j_1}$ and $C$, if the latter exist.

	\item For $k \geq  2$, the entries of $x$ corresponding to $A_{i_k}, A_{j_k}$, $B_{i_k}$ and $B_{j_k}$ are obtained in the same manner as $A_{i_1}, A_{j_1}$, $B_{i_1}$ and $B_{j_1}$.
	
	\item All other entries are zero.
\end{enumerate}
 
\noindent Then $x \in \Or_{\mu}$.
 
\end{theorem}

\begin{example} We illustrate the symmetrising algorithm via the orbit corresponding to the partition $\lambda = (4,4,2)$ in $\gl\l_{10}(k)$. Since all parts of $\lambda$ are even, $I= \emptyset$, therefore we may skip straight to Subsection \ref{x_J}. We have three blocks as in the following illustration, $A_{j_1}, C = A_{j_0} = B_{j_0}$, and $B_{j_1}$. (So $a_{1,1}$ is in the $1,3$ position of $x$ and $b_{3,2}$ is in the $8,10$ position.) Recall that the initial element $x$ was chosen to be in $\Or_{\lambda}$ and the specially constructed open set $S$. In step $(a)$ we have performed Gauss-Jordan elimination to put $C$ into the desired form. Note that $A_{j_1}$ and $B_{j_1}$ will still have maximal rank because of the way we chose $S$. In step $(b)$ we have performed column operations to obtain the desired form on $A_{j_1}$. Recall that every time we perform a column operation on $A_{j_1}$ this induces a row operation on $C$, which is then put back into the desired from by a suitable column operation on $C$. $B_{j_1}$ may change in the process, but its rank will remain maximal. Thus, in step $(c)$ we may use row operations on $B_{j_1}$ (remember that the rows of $B_{j_1}$ correspond to columns of the ambient matrix) to obtain the penultimate array shown. Finally, in step $(d)$ we have subtracted $b''_{3,2} \times$ column $3$ from column $1$ and we have subtracted $b''_{3,1} \times$ column $2$ from column $1$ in $B_{j_1}$ to obtain the desired form. Of course these will filter through $C$ again, but the effect on $A_{j_1}$ will be benign --- we will have added multiples of column 1 to columns 2 and 3, but column 1 consists of zeros.

{ \tiny \[  x = \begin{matrix}
 \ \ \ & \ \ \ &a_{1,1}&a_{1,2}&a_{1,3}&&&&&  \\
    &&a_{2,1}&a_{2,2}&a_{2,3}&&&&&  \\
      &&&&&c_{1,1}&c_{1,2}&c_{1,3}&&  \\
        &&&&&c_{2,1}&c_{2,2}&c_{2,3}&&  \\
          &&&&&c_{3,1}&c_{3,2}&c_{3,3}&&  \\
            &&&&&&&&b_{1,1}&b_{1,2}  \\
              &&&&&&&&b_{2,1}&b_{2,2}  \\
                &&&&&&&&b_{3,1}&b_{3,2}  \\
                  &&&&&&&&&  \\
                    &&&&&&&&&
 \end{matrix}\ \ \]}

{ \tiny \[ \stackrel{(a)}{\longrightarrow}\ \ \begin{matrix}
 \ \ \ & \ \ \ &a'_{1,1}&a'_{1,2}&a'_{1,3}&&&&&  \\
    &&a'_{2,1}&a'_{2,2}&a'_{2,3}&&&&&  \\
      &&&&&0&0&1&&  \\
        &&&&&0&1&0&&  \\
          &&&&&1&0&0&&  \\
            &&&&&&&&b'_{1,1}&b'_{1,2}  \\
              &&&&&&&&b'_{2,1}&b'_{2,2}  \\
                &&&&&&&&b'_{3,1}&b'_{3,2}  \\
                  &&&&&&&&&  \\
                    &&&&&&&&&
 \end{matrix} 
\]}

{ \tiny \[  \stackrel{(b)}{\longrightarrow}\ \ \begin{matrix}
 \ \ \ & \ \ \ &0&0&1&&&&&  \\
    &&0&1&0&&&&&  \\
      &&&&&0&0&1&&  \\
        &&&&&0&1&0&&  \\
          &&&&&1&0&0&&  \\
            &&&&&&&&b''_{1,1}&b''_{1,2}  \\
              &&&&&&&&b''_{2,1}&b''_{2,2}  \\
                &&&&&&&&b''_{3,1}&b''_{3,2}  \\
                  &&&&&&&&&  \\
                    &&&&&&&&&
 \end{matrix} \ \ \stackrel{(c)}{\longrightarrow}\ \  \begin{matrix}
 \ \ \ & \ \ \ &0&0&1&&&&&  \\
    &&0&1&0&&&&&  \\
      &&&&&0&0&1&&  \\
        &&&&&0&1&0&&  \\
          &&&&&1&0&0&&  \\
            &&&&&&&&0&1  \\
              &&&&&&&&1&0  \\
                &&&&&&&&b'''_{3,1}&b'''_{3,2}  \\
                  &&&&&&&&&  \\
                    &&&&&&&&&
 \end{matrix} \ \ \]}

{ \tiny \[ \stackrel{(d)}{\longrightarrow}\ \  \begin{matrix}
 \ \ \ & \ \ \ &0&0&1&&&&&  \\
    &&0&1&0&&&&&  \\
      &&&&&0&0&1&&  \\
        &&&&&0&1&0&&  \\
          &&&&&1&0&0&&  \\
            &&&&&&&&0&1  \\
              &&&&&&&&1&0  \\
                &&&&&&&&0&0  \\
                  &&&&&&&&&  \\
                    &&&&&&&&&
 \end{matrix} 
\]}

\end{example}

\begin{rmk} In order to symmetrise pairs of blocks described by Figure \ref{centralpair4}, we implicitly assumed that $2 \not= 0$. In fact, this is crucial since, when the characteristic is not $2$, the matrix {\footnotesize \[\left( \begin{array}{cccc}  \label{char2}
\ \ \   & \ 1\  & \ 1\  &     \\
\    &   &   & \ 1\    \\
\    &   &   & 1   \\
\    &   &   &     \end{array} \right)\]}

\noindent is the symmetric canonical form corresponding to $(3,1)\vdash 4$. (Moreover, one may check that any symmetric matrix in $\gl_2$ (defined by $(3,1)$) is in the same orbit as the above.) However, in characteristic $2$, it has Jordan form $(2,2)$. Using a computer we have also found similar examples for $n= 5,6$ and $7$ in characteristic $2$. \end{rmk}

\section{The symplectic and orthogonal Lie algebras}

We now consider the other classical algebras. More precisely, let $G$ be a simple classical algebraic group of Type $\Btype_l$, $\Ctype_l$ or $\Dtype_l$, together with the adjoint action on the nilpotent variety $\gl_{\nil}$ of its Lie algebra $\gl$. By considering the natural matrix representation of $\gl$ we may compute the set of elementary divisors of an element of each orbit, thus defining a partition of $2l$ for groups of Type $\Ctype_l$ and $\Dtype_l$ and of $2l+1$ for groups of Type $\Btype_l$. Letting $n=2l$ or $2l+1$ accordingly, the corresponding map \[\{\mbox{ orbits of $\gl$ }\} \longrightarrow \mathcal{P}(n)\]

\noindent is an injection if $G$ is of Type $\Btype_l$ or $\Ctype_l$, while for a group of Type $\Dtype_l$ very even partitions (i.e. those consisting of only even parts, each having even multiplicity) correspond to two orbits. If $G$ is of Type $\Btype_l$ or $\Dtype_l$ then the image consists precisely of those partitions in which even parts occur with even multiplicity, while the image for Type $\Ctype_l$ consists of those partitions in which odd parts occur with even multiplicity. We shall refer to the partitions not in the image as {\em bad}. Using this classification of orbits we may use the notation $\Or_{\mu}$ to denote an orbit corresponding to a partition $\mu$. (It is customary, when $\mu$ is very even, to denote the two orbits by $\Or_{\mu}', \Or_{\mu}''$.)

Let $M$ be an $n \times n$ matrix over $k$. Then the set of $n \times n$ matrices $X$ over $k$ satisfying the condition \begin{equation} \label{Liecond} X^TM + MX = 0,\end{equation} \noindent is a Lie algebra under the commutator operation. We construct the classical algebras by selecting a suitable $M$, following the standard text of Carter \cite{Car}. (But note that our choices of $M$ differ slightly from those in \cite{Car}.)

By considering the restriction of $\GL_n(k)$-orbits on $\gl\l_n(k)$ we will show how to obtain a canonical representative for each of these orbits by modifying slightly the symmetric canonical form from Theorem \ref{x_J} so that it satisfies (\ref{Liecond}). 

Clearly it is impossible to modify a canonical element $x$ (as in Theorem \ref{x_J}) corresponding to a bad partition so that it satisfies (\ref{Liecond}). In view of this we start by observing a characterisation of bad partitions in terms of a feature of the block structure of $\gl_2$. Then, assuming the absence of this feature we present a sequence of elementary operations on $x$ so that $x$ satisfies (\ref{Liecond}). By Lemma \ref{SpSt} this characterisation is as follows. For Type $\Ctype_n$ the bad partitions correspond to the existence of a block in $I$ with an odd number of rows. For Types $\Btype_n$ and $\Dtype_n$ they correspond to the existence of a block in $J$ with an odd number of rows.

\subsection{Type $\Ctype_l$} Here we let \begin{equation} \label{C_nM} M = \left(\begin{array}{cc}
 & J_l \\
-J_l & \end{array} \right).\end{equation}

\noindent Then, writing \[X = \left(\begin{array}{cc}
X_{11} & X_{12} \\
X_{21} & X_{22} \end{array} \right),\]

\noindent in terms of $l \times l$ blocks, $X \in \s\pl_{2l}(k)$ if, and only if, $X_{12} = f(X_{12})$, $X_{21} = f(X_{21})$ and $X_{11} = -f(X_{22})$.

The canonical form $x$ from Theorem \ref{x_J} already satisfies the conditions on $X_{12}$ and $X_{21}$. It suffices, therefore, to change all non-zero entries of $X_{22}$ from $1$ to $-1$. Now the absence of bad partitions means that blocks of the form Figure \ref{centralpair5} can not occur, and so there is at most one $1$ in each row. It follows that we can rescale the non-zero entries of $x$ independently using row operations and thus obtain the desired form. This is illustrated in Table 6.

\subsection{Type $\Dtype_l$} Now let $M = J_{2l}$. Then, writing \[X = \left(\begin{array}{cc}
X_{11} & X_{12} \\
X_{21} & X_{22} \end{array} \right),\]

\noindent in terms of $l \times l$ blocks, $X \in \s\ol_{2l}(k)$ if, and only if, $X_{12} = -f(X_{12})$, $X_{21} = -f(X_{21})$ and $X_{11} = -f(X_{22})$.

This time more work is required since rescaling alone will not be sufficient to satisfy the condition on $X_{12}$, if $C$ exists, as it will have non-zero entries fixed by $f$. We therefore begin by obtaining a new form for $C$ such that $x_C = -f(x_C)$. The new $x_C$ will still be a permutation matrix, although it will now have all entries on the diagonal fixed by $f$ equal to 0. First observe that we may perform operations on the top $k_1 - k_2$ columns of the existing $x_C$ without changing any other entry of $x$. We may therefore reverse the order of these columns. Similarly, one may perform operations on the bottom $k_2 - k_3$ columns of $x_C$ without changing any other entry of $x$, via the mirror afforded by $B_{j_1}$. Hence, we may reverse the order of these columns too. We continue this process until we have transformed $x_C$ into a a matrix with copies of (various sized) identity matrices lined up along the diagonal fixed by $f$, with zeros elsewhere, and the rest of $x$ left unchanged. We may now rescale some of the entries of $x_C$ to $-1$ so that $x_C = -f(x_C)$, {\em provided} that each of the numbers $k_1 - k_2, k_2 - k_3, k_3 - k_4, \dots$ is even, i.e. provided that $\mu$ is not bad.

To finish one just rescales all non-zero entries of the $B$-blocks from $1$ to $-1$. For this we simply multiply all rows by $-1$ on $B_{i_1}, B_{i_3}, B_{i_5}, \dots  $ and $B_{j_1}, B_{j_3}, B_{j_5}, \dots  $.

\begin{rmk} In the case that $\mu$ is very even this only corresponds to one of the two orbits associated to $\mu$. \end{rmk}

\subsection{Type $\Btype_l$} We let $M = J_{2l+1}$, and write \[X = \left(\begin{array}{ccc}
X_{11} & X_{12} & X_{13} \\
X_{21} & X_{22} & X_{23} \\
X_{31} & X_{32} & X_{33} \end{array} \right),\]

\noindent where $X_{11}, X_{13}, X_{31}$ and $X_{33}$ are $l \times l$ matrices, $X_{12}$ and $X_{32}$ are $l \times 1$ matrices, $X_{21}$ and $X_{23}$ are $1 \times l$ matrices, and $X_{22}$ is a $1 \times 1$ matrix. Hence, $X \in \s\ol_{2l+1}(k)$ if, and only if, $X_{13} = -f(X_{13})$, $X_{31} = -f(X_{31})$, $X_{22} = 0$, $X_{11} = -f(X_{33})$, $X_{21} = -f(X_{32})$, and $X_{12} = -f(X_{23})$. 

The canonical form is obtained in exactly the same manner as for Type $\Dtype_l$.

\section{Springer morphisms and unipotent conjugacy} \label{springersection}

In this section we explain how to compute canonical unipotent elements corresponding to the nilpotent ones we obtained previously. Let $G$ be any connected reductive group, defined over the field with $q$ elements, where $q$ is a power of a good prime for $G$. Let $F$ be the corresponding Frobenius endomorphism. Then there exists an $\Ad(G)$-compatible Frobenius endomorphism on $\gl=\Lie(G)$, which we also denote by $F$. I.e. $F(g\cdot x) = \nolinebreak F(g)\cdot \nolinebreak F(x)$ for $g\in G$, $x \in \gl$, where $\cdot$ denotes the adjoint action. (E.g. if $G$ is a classical group and $F = F_q$ --- the map which acts by raising entries in the natural representation of $G$ to the $q^{th}$-power --- then the map defined in the same way on $\gl$ is compatible with $F$.) Now let $G= \GL_n(k)$ or one of the groups \[\{x \in \GL_n(k) \ | \  x^TMx = M \}, \]

\noindent where $M$ is as in the previous chapter. In this setting there exists a $G$-equivariant bijective morphism, a Springer morphism,  \[ s: G_{\uni} \longrightarrow \gl_{\nil}\]

\noindent which commutes with the $F$-actions [\cite{SpSt}, Theorem III.3.12]. In fact, we can write out such a map explicitly for classical groups following \cite{Kawa}. For $G=\GL_n(k)$, together with $F_q$, we may take Springer's morphism to be $x \mapsto x - 1$. For Types $\Btype$, $\Ctype$ and $\Dtype$, with untwisted Frobenius endomorphisms, it is easy to check that the Cayley map $x \mapsto (x-1)(x+1)^{-1}$ works.

\subsection{Finite unitary groups}  \label{unitarysection} Assume, initially, that $k$ is a field of characteristic at least $3$. We now turn our attention to the unipotent elements in the finite groups $G^F$. When $G$ has a disconnected centre, or is of Type $\Btype, \Ctype$ or $\Dtype$, the $F$-stable unipotent classes may split into several $G^F$-orbits and so there are more unipotent conjugacy classes in $G^F$ than in $G$. Using a Springer morphism we may therefore map our nilpotent representatives from $\gl$ (noting that they are all $F$-stable) to obtain some unipotent representatives in $G^F$, but more work would be needed to obtain a {\em full} set of representatives. However, if $G=\GL_n(k)$ then this splitting does not occur for any Frobenius endomorphism. Hence we may compute a full set of representatives for the unipotent classes of $G^F$ in this case. In the case of finite general linear groups then we simply take the elements $x+1 \in G^F=\GL_n(\F_q)$ where $x$ varies over the symmetric canonical forms from Theorem \ref{x_J}. Alternatively, the Jordan canonical form also affords a perfectly good set of representatives in this case. The author believes, though, that in the case of the finite unitary groups $G^F=\GU_n(\F_q)$ (i.e. those afforded by a twisted Frobenius endomorphism) no canonical form for unipotent elements was known until now.

We will use the following twisted Frobenius endomorphism on $G=\GL_n(k)$. For $(g_{i,j}) \in \nolinebreak G$, let \begin{equation}   \label{unitaryF}   F((g_{i,j})) = (g_{n+1-j,n+1-i}^q)^{-1}.\end{equation} 

\noindent We will use the compatible Frobenius endomorphism on $\gl= \gl\l_n(k)$ given by \begin{equation}  \label{unitaryFLie}  F((g_{i,j})) = (g_{n+1-j,n+1-i}^q),\end{equation}

\noindent for $(g_{i,j}) \in \gl$. We also note that the map given by \begin{equation}  \label{unitaryFLie2}  F^-((g_{i,j})) = -(g_{n+1-j,n+1-i}^q),\end{equation}

\noindent for $(g_{i,j}) \in \gl$ is also commonly used, and may be more convenient in certain situations. Naturally, we have chosen to use (\ref{unitaryFLie}) as it fixes the representatives obtained in Theorem \ref{x_J}. 

\begin{proposition}  Let $\al \in \F_{q^2} \setminus \F_q$. Then the map \begin{equation}  \label{unitarySpringer}  s: g \mapsto (g -1) (\al - \al^q g)^{-1},\end{equation} with inverse \[s^{-1} : x \mapsto (1 + \al^q x)^{-1} (\al x + 1),\]

\noindent  is a $G$-equivariant bijective morphism $G_{\uni} \rightarrow \gl_{\nil}$ which commutes with {\em(\ref{unitaryF})} and {\em (\ref{unitaryFLie})}.

\end{proposition}

\begin{proof} The only non-trivial thing to check is that for all $g \in G_{\uni}$, $s(F(g)) = F(s(g))$. This can be checked explicitly be writing these Frobenius endomorphisms in terms of familiar matrix operations as follows. We have \[F(g) = J_nF_q(g^{-T})J_n,\] 

\noindent for $g \in G_{\uni}$, where $F_q$ denotes the standard $q^{th}$-power Frobenius endomorphism, and \[F(x) = J_nF_q(x^{T})J_n.\] 

\noindent for $x \in \gl_{\nil}$.\end{proof}

\begin{corollary} If $x \in \gl_{\nil}$ is $F$-stable, then $s^{-1}(x) \in G^F$. \end{corollary}

\begin{rmk} If one uses {\em (\ref{unitaryFLie2}) }, the Cayley map is suitable if the characteristic is not $2$; Kawanaka presents a map in \cite{Kawa} suitable for any characteristic. \end{rmk}

We may now compute representatives for unipotent classes in the finite general unitary groups, by applying the map $s^{-1}$ from Proposition \ref{unitarysection} to the canonical forms described by Theorem \ref{x_J}, provided $\chara(k) \not= 2$. We may adapt these representatives slightly to obtain a canonical set of representatives valid for arbitrary characteristic. Recall that the algorithm used in the proof Theorem \ref{x_J} fails in characteristic $2$ when trying to pass from a situation described by Figure \ref{centralpair4} to one described by Figure \ref{centralpair5}. In fact symmetry, whilst remaining in $\gl_2$, is impossible in some situations, as we saw in Remark \ref{x_J}.

However, instability under $f$ need not obstruct stability under $F = F_q \circ f$. We may obtain mere $F$-stability in all cases as follows, transforming Figure \ref{centralpair4} to Figure \ref{centralpair6}.

\begin{enumerate}

	\item Choose $\al \in \F_{q^2}\setminus \F_q$.
	
		\item Multiply row $(l_2+1)/2$ of $B_{i_1}$ by $1 + \al^{q-1}$.
		
			\item Multiply row $(l_2+1)/2$ of $A_{i_1}$ by $\al$.
	
	\item Add $\al^q(1 + \al^{q-1})^{-1}$ times column $l_1 + 1 - (l_2+1)/2$ to column $(l_2 + 1)/2$ in $B_{i_1}$.
	
    \item Add $\al^{-1}$ times column $l_1 + 1 - (l_2+1)/2$ to column $(l_2 + 1)/2$ in $A_{i_1}$.
	
	    \item One may also need to rescale the central rows of $A_{i_2}, A_{i_3}, A_{i_4},\dots$ depending on the parity of the $l_1, l_2, l_3, \dots$.
\end{enumerate}

\begin{figure}[ht]		\centering

\includegraphics{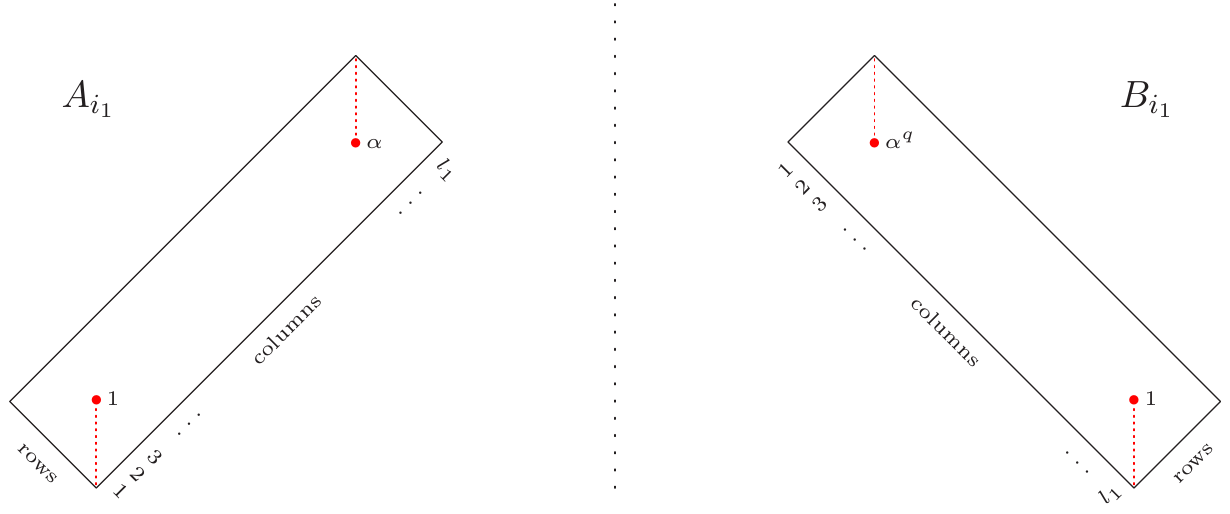}

\caption{$F$-stable central pair of blocks in $x_I$ ($l_2$ odd)} 

\label{centralpair6}

\end{figure}

In the interest of a generic approach we will use the $F$-stable form of Figure \ref{centralpair6} even when $\chara(k) \not= 2$ in the tables that follow. Tables of the symmetric canonical form for $n=2, 3, 4$ and $5$ are presented, together with the corresponding forms for the other classical Lie algebras and finite unitary groups. (Note that for an orbit corresponding to a very even partition in Type $\Dtype$ a representative of only one of the two orbits is given.)

\begin{table}
		\centering

\caption{Canonical forms for $n=2$}

\includegraphics{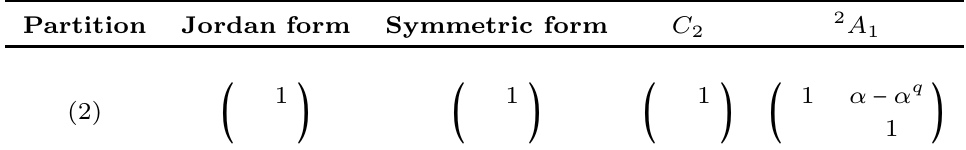} 

\end{table}

\begin{table}
		\centering

\caption{Canonical forms for $n=3$}

\includegraphics{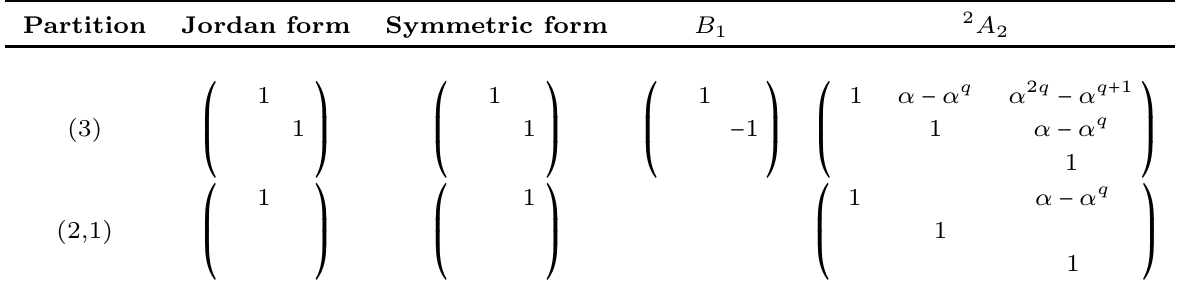} 

\end{table}

\begin{figure}
		\centering

\begin{sideways}

\includegraphics{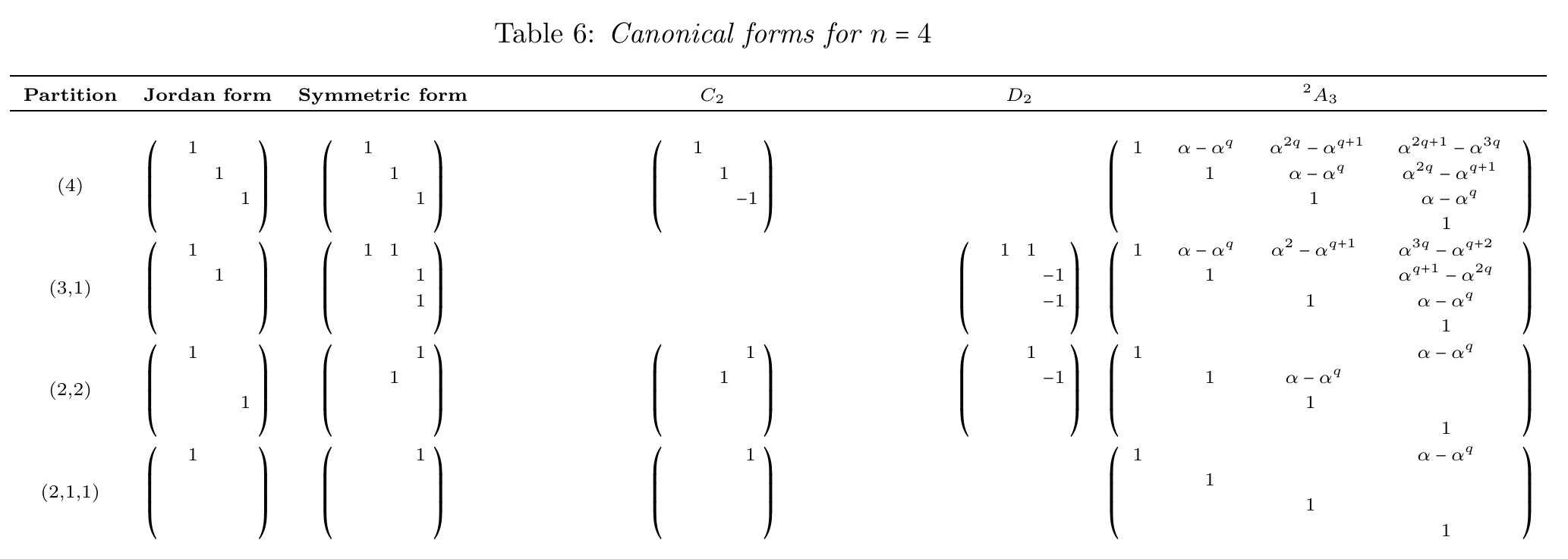} 
\end{sideways}

\end{figure}

\begin{figure}
		\centering

\begin{sideways}

\includegraphics{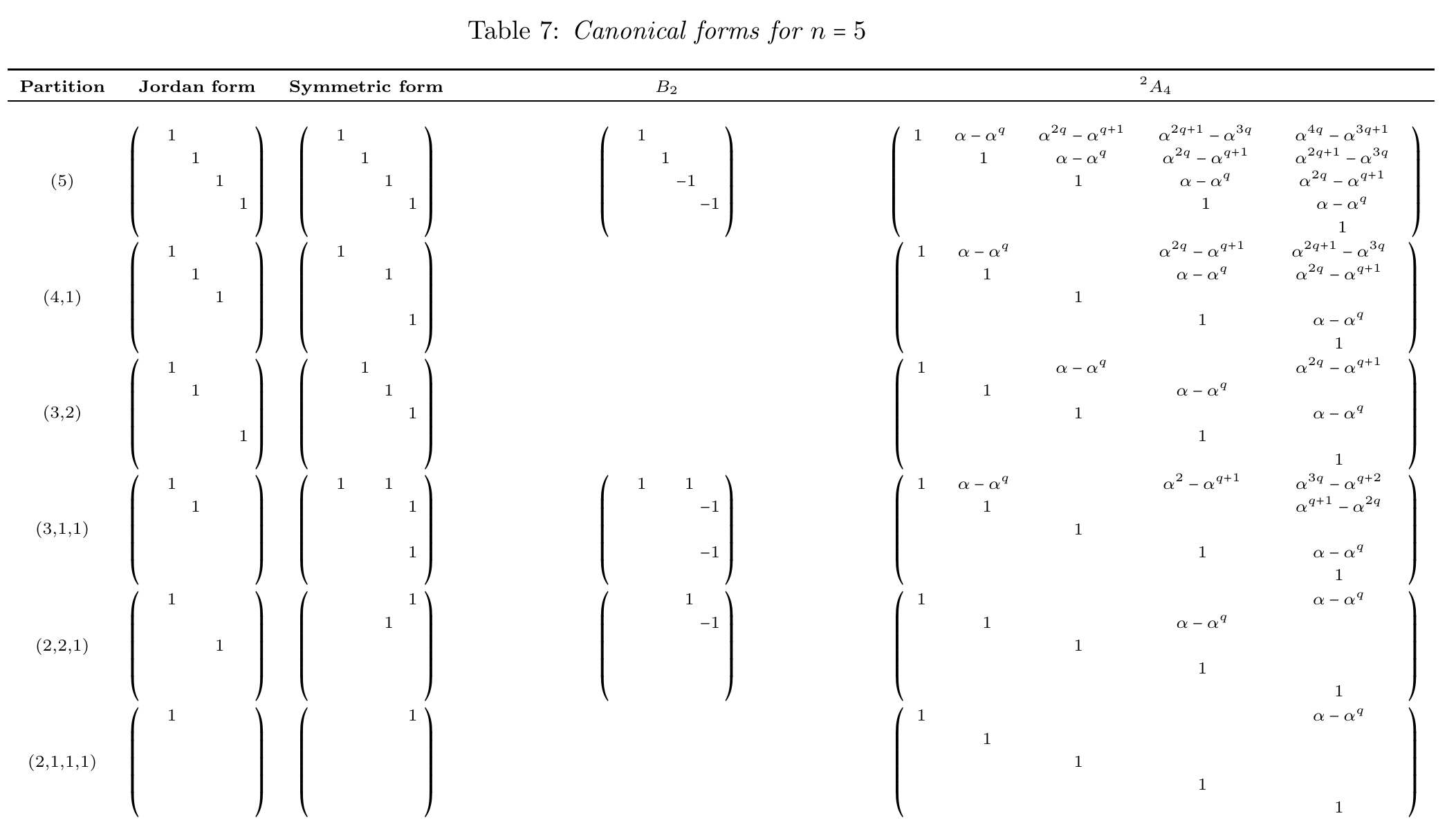} 

\end{sideways}

\end{figure}

\begin{acknowledgements} The author would like to thank the anonymous referee for many useful suggestions for improving the original manuscript. He also wishes to express his gratitude for the hospitality and guidance of Professor T. Shoji and to Nagoya University in the summer of 2009, during which some of this work was completed as part of a visit to study generalised Gelfand-Graev representations of finite groups of Lie type. In addition, he thanks the Japan Society for the Promotion of Science (JSPS) for funding his trip. \end{acknowledgements}

\end{document}